\documentclass[12pt, reqno]{amsart}

\usepackage{amsfonts,amsmath,amsthm,amssymb}
\usepackage{latexsym}

\begin{document}

{\theoremstyle{plain}
  \newtheorem{theorem}{Theorem}[section]
  \newtheorem{corollary}[theorem]{Corollary}
  \newtheorem{proposition}[theorem]{Proposition}
  \newtheorem{lemma}[theorem]{Lemma}
  \newtheorem{question}[theorem]{Question}
  \newtheorem{conjecture}[theorem]{Conjecture}
  \newtheorem{claim}[theorem]{Claim}
}

{\theoremstyle{definition}
  \newtheorem{definition}[theorem]{Definition}
  \newtheorem{remark}[theorem]{Remark}
  \newtheorem{example}[theorem]{Example}
}

\numberwithin{equation}{section}
\def\Q{\mathbb{Q}}
\def\Z{\mathbb{Z}}
\def\N{\mathbb{N}}
\def\R{\mathbb{R}}
\def\C{\mathbb{C}}

\def \a {\alpha}
\def \aa {\overline{\alpha}}
\def \b {\beta}
\def \bb {\overline{\beta}}
\def \d {\delta}
\def \e {\epsilon}
\def \l {\lambda}
\def \f {\phi}
\def \g {\gamma}
\def \G {\Gamma}
\def \H {\mathcal{H}}
\def \p {\pi}
\def \n {\nu}
\def \m {\mu}
\def \s {\sigma}
\def \S {\Sigma}
\def \w {\omega}
\def \F {\Phi}
\def \O {\mathcal{O}}
\def \W {\Omega}
\def \k {\bold{k}}
\def \bp {\bold{p}}
\def \U {\mathcal{U}}
\def \t {\tau}
\def\D {\mathcal{D}}
\def\X {\mathcal{X}}
\def\cR {\mathcal{R}}
\def\S {\mathcal{S}}
\def\NN{\mathcal{N}}
\def \P {\mathcal{P}}
\def \T {\mathcal{T}}
\def \R {\mathbb{R}}
\def \N {\mathbb{N}}
\def \Z {\mathbb{Z}}
\def \A{\mathcal{A}}
\def \et {\eta}
\def \th {\theta}
\def \X {\mathcal{X}}
\def \Y {\mathcal{Y}}
\def\k{\mathbf{k}}

\def \ra {\rightarrow}

\title{Generating sequences of rank 1 valuations} 
\author{Olga Kashcheyeva}
\address{University of Illinois at Chicago, Department of Mathematics,
Statistics and Computer Science, 851 S. Morgan (m/c 249), Chicago,
IL 60607, USA} \email{kolga@uic.edu}

\begin{abstract} We consider a rank 1 valuation centered in a regular local ring. Given an algebraic function field $K$ of transcendence degree 3 over $\k$, a regular local ring $R$  with $QF(R)=K$  and a rational $\k$-valuation $\n$ of $K$, we provide an algorithm for constructing  a generating sequences for $\n$ in $R$. This is a generalization of the algorithm of \cite{K} where only valuations of rational rank 1 were considered. We then give a comparison example of a rational rank 3 valuation centered in a ring of polynomials in three variables for which we explicitly compute a set of key polynomials and a generating sequence.
\end{abstract}

\maketitle

\section{Introduction}

The graded algebra $gr_{\n} R$ associated to a valuation $\n$ centered in a domain $R$ encodes important geometric properties of the valuation. It plays a central role in the approach to local uniformization developed by Teissier (\cite {T1}, \cite {T2}),  and it is investigated further in \cite{CMT}.  A generating sequence of a valuation is a set of elements of $R$ whose images generate $gr_{\n} R$. Explicit constructions of generating sequences of rational valuations centered in 2-dimensional regular local rings together with their connection to sequences of quadratic transforms along the valuation can be found in \cite{Sp}, \cite{CP}, \cite{GHK} . In this paper we work with rank 1 rational valuations centered in 3-dimensional regular local rings. 

More precisely, let $\k$ be a field and $K$ be an algebraic function field of transcendence degree 3 over $\k$. Let 
$\n$ be a rank 1 $\k$-valuation on $K$ with valuation ring $(V,m_V)$, value group $\G_{\n}\subseteq \R$, and residue field $V/m_V=\k$.  Let $(R, m_R)$ be a local subring of $K$ with $\k\subseteq R$ and $QF(R)=K$. We assume that $R$ is dominated by $\n$, that is $R\subseteq V$ and $R\cap m_V=m_R$. We also assume that $R$ is regular.  Recall that $S_R = \nu(R \backslash \{0\})$ denotes the value semigroup of $R$, and for  $\s \in S_R$, $I_{\s}=\{f\in R \mid \ \n(f)\ge \s\}$  denotes the corresponding valuation ideal of $R$. Since $R$ is Noetherian, the semigroup $S_R$ is well ordered, and an element $\s\in\S_R$ has a unique successor $\s_+\in S_R$ such that $I_{\s_+}=\{f\in R \mid \ \n(f)> \s\}$. In this set up the associated graded algebra of $R$ along $\n$ is 
$$gr_{\n}(R)=\bigoplus_{\s\in\S_R}I_{\s}/I_{\s_+} $$
The degree zero subring of $gr_{\n}(R)$ is isomorphic to $\k$, and  $gr_{\n}(R)$ is an algebra over $\k$ that is isomorphic to the semigroup algebra of the value semigroup $S_R$. It is a domain which is generally not Noetherian. It can be represented as the quotient of a polynomial algebra by a prime binomial ideal (see \cite{T2} ).

A collection $\{ Q_i \}_{i\in I}$ of elements of $R$ is a {\it generating sequence} of $\n$  if for every $\s\in S_R$ the ideal $I_\s$ is generated by the set $\{\prod _{i\in I}Q_i^{b_i} \mid b_i \in \N_0, \s\le \sum_{i\in I} b_i\n(Q_i)<\infty\}$. 
A generating sequence of $\n$ is minimal if none of its proper subsequences is a generating sequence of $\n$.

In section \ref{construction} we provide an algorithm for constructing generating sequences of rank 1 valuations. This is a generalization of the algorithm of \cite{K} where only rational rank 1 valuations were considered. In the construction we obtain the sequence of jumping polynomials $\{P_j\}_{j>0}\cup\{T_j\}_{j>0}$. Using the results of section \ref{properties1} we then show that $\{P_j\}_{j>0}\cup\{T_j\}_{j>0}$ is a generating sequence of valuation in section \ref{generating seq}. The explicit formulas of section \ref{construction} provide a presentation by generators and binomial relations for the associated graded algebra $gr_{\n} R$.

\section{Notation and definitions}

We denote by $\N$ the set of all positive integers, and we denote by $\N_0$ the set of all nonnegative integers. Let $n\in\N$ be a fixed positive integer. If $A\in \N_0^n$ denote by $a_1,a_2,\dots,a_n$ the components of $A$, so that $A=(a_1,a_2,\dots, a_n)$ with $a_1,a_2,\dots,a_n\in\N_0$. If $k\in\N$ is such that $k> n$ and $A\in \N_0^n$ by abuse of notation we also say that $A\in  \N_0^k$ with $a_i=0$ for all $i>n$. Moreover, we consider the set of infinite sequences of nonnegative integers with finitely many nonzero terms $\NN=\bigcup_{k=1}^{\infty} \N_0^k$, and for $A\in  \N_0^n$ we say that $A=(a_1,a_2,\dots, a_n, 0,0,\dots)$ as an element of  $\NN$. Also, for $n,m\in\N$ and $(A,C)\in\N_0^n\times\N_0^m$ we say that $(A,C)$ is an element of $\NN\times\NN$ as $A\in\NN$ and $C\in\NN$, while $AC=(a_1,a_2,\dots,a_n,c_1,c_2,\dots,c_m)$ is an element of $\N_0^{n+m}$. We denote by $E_n\in\N^n$ the vector with the $n-th$ entry equal to 1 and all other entries equal to 0, and we denote by $\bar{0}\in\N^n$ the vector with all entries equal to 0.

\begin{definition}
 For two elements $A,B\in\NN$ we say that $A$ is less than or equal to $B$, denoted $A\preceq B$, if $a_i\le b_i$ for all $i\ge 1$.  We say that $A$ is strictly less than $B$, denoted $A\prec B$, if $A\preceq B$ and $A\neq B$. 
\end{definition}

Let $\{\a_j\}_{j>0}$ be a possibly infinite sequence of nonnegative real numbers. We denote by $\a(n)$ the vector $(\a_1,\a_2,\dots,\a_n)$ and we denote by $\a$ the vector $(\a_1,\a_2,\dots,\a_n,\dots)$. We consider the additive group $G_n$ and the additive semigroup $S_n$ generated by $\a_1,\a_2,\dots,\a_n$:
$$ G_n =\sum_{i=1}^n\a_i\Z \quad\quad\quad {\text {and}} \quad\quad\quad  S_n =\sum_{i=1}^n\a_i\N_0$$
 We also write $G_n=<\a(n)>$ and $S_n=[\a(n)]$. If $A\in\NN$ we denote the linear combination $a_1\a_1+\dots+a_n\a_n$ by $A\cdot\a(n)$. 
 
A real number $\s$ is said to be commensurable with a subgroup $G\subseteq\R$ if $\s\in\Q\otimes_\Z G$, or equivalently, if there exists $s\in\N$ such that $s\s\in G$.  We say $\a_{n+1}$ is commensurable with $\a(n)$ if $\a_{n+1}$ is commensurable with $G_n=<\a(n)>$. We say that  $A\in\NN$ pushes $\a_{n+1}$  into the semigroup $S_n$ if $a_{n+1}>0$ and $A\cdot\a(n+1) \in S_n$. Observe that, provided $\a_{n+1}$ is commensurable with $\a(n)$, there always exists $A\in\N_0^{n+1}$ that pushes $\a_{n+1}$ into $S_n$. Indeed, since $\a_{n+1}$ is commensurable  with $\a(n)$ there exist $s\in \N$ and $b_1,b_2,\dots,b_n\in \Z$ such that $s\a_{n+1}=b_1\a_1+b_2\a_2+\dots+b_n\a_n$. For $1\le i\le n$ set $a_i=-b_i$ if $b_i\le 0$, and set $a_i=0$ if $b_i>0$. Then $A=(a_1,a_2,\dots,a_n,s)$ pushes $\a_{n+1}$ into $S_n$. Conversely, if there exists  $A\in\NN$ that pushes $\a_{n+1}$ into $S_n$ then $\a_{n+1}$ must be commensurable with $\a(n)$. Indeed, since  $A\cdot \a(n+1)\in S_n$ it follows that $a_{n+1}\a_{n+1}\in G_n$ with $a_{n+1}\in\N$.

\begin{definition} 
For a sequence of nonnegative real numbers  $\{\a_j\}_{j>0}$ let $S_n=[\a(n)]$, and  let $\a_{n+1}$ be commensurable with $\a(n)$. We say that $A\in\NN$ is {\it  reduced} with respect to $\a(n+1)$ if $A$ pushes $\a_{n+1}$ into $S_n$ and no $B\in\NN$ that is strictly less than $A$ pushes $\a_{n+1}$ into $S_n$. 
\end{definition}

\begin{definition}
If $A,C\in\NN$ and $\T\subseteq \NN\times\NN$ we say that $(A,C)$ is {\it irreducible} with respect to $\T$ if $(B,D)\notin \T$ for all $B\preceq A$ and $D\preceq C$.
\end{definition}

Let $i\in\{1,\dots,n\}$ be a fixed index. Denote by $\pi_i: \N_0^n\ra \N_0$ the projection onto $i$-th component: $\pi_i(A)=a_i$. In particular, for two elements $A,B\in\NN$ we have $A\preceq B$ if and only if $\pi_i(A)\le\pi_i(B)$ for all $i\ge 1$.  Denote by $\pi_{\hat{i}}: \N_0^n\ra \N_0^{n-1}$ the projection onto $(n-1)$-tuple that excludes  the $i$-th component: $\pi_{\hat{i}}(A)=(a_1,\dots,a_{i-1},a_{i+1},\dots,a_n)$.

If $x_1,x_2,\dots,x_n,\dots$ are variables and $A\in\N_0^n$, $x^A$ denotes the monomial $x_1^{a_1}x_2^{a_2}\cdots x_n^{a_n}$.

\section{Construction of jumping polynomials}\label{construction}

We assume that $\k$ is a field and $K$ is an
algebraic function field of transcendence degree 3 over $\k$.
$\n$ is a rank 1 $\k$-valuation on $K$ with valuation ring $(V,m_V)$,
value group $\G\subset \R$, and residue field $V/m_V=\k$. 
$(R, m_R)$ is a local subring of $K$ with $\k\subset R$
and $R_{(0)}=K$. We assume that $R$ is dominated by $\n$, that is $R\subset V$ and $R\cap m_V=m_R$.  
We also assume that 
 $R$ is a regular ring with regular parameters $x,y,z$ and $\n(x)\le\n(y)\le\n(z)$. 

To construct a generating sequence of $\n$ in $R$ we define
the set of jumping polynomials $\{P_j\}_{j>0}\cup\{T_j\}_{j>0}$ in $R$.  

Let $P_1= x,\; P_2=y$ and  $\P_1=\P_2=\emptyset$. For all $i\ge 1$
we set $\b_i=\n(P_i)$ and $G_i=<\b(i)>$,  $S_i=[\b(i)]$.

For $i>1$, if $\b_i$ is commensurable with $\b(i-1)$, set $q_i=\min\{q\in\N\mid q\b_i\in G_{i-1}\}$.
 Let $L(i)\in\N_0^{i-1}$ be such that $(L(i),\bar{0})$
is irreducible with respect to $\P_i$ and $q_i\b_i=L(i)\cdot\b(i-1)$. Denote by $\l_i$ the residue of
$P_i^{q_i}/P^{L(i)}$ in $V/m_V$ and set $P_{i+1}=P_i^{q_i}-\l_iP^{L(i)}$. Also, set $\P_{i+1}=\P_i\cup\{(q_iE_i,\bar{0})\}$. 

For $i>1$, if $\b_i$ is not commensurable with $\b(i-1)$, set $q_i=\infty$. The sequence $\{P_j\}_{j>0}$ is finite in this case with $P_i$ being the last element.

Finally set $G_0=\{0\}$, $q_1=\infty$, $G=\bigcup_{j>0}G_j$,  $S=\bigcup_{j>0} S_j$
and $\P=\bigcup_{j>0}\P_j$.

\begin{remark}
The sequence $\{P_j\}_{j> 0}$ is well defined since for every $j\ge 2$ such that  $\b_j$ is commensurable with $\b(j-1)$ there exists a unique $L(j)\in\N_0^{j-1}$ such that $(L(j),\bar{0})$ is irreducible with respect to $\P_j$ and $q_j\b_j=L(j)\cdot\b(j-1)$. 
 (See Corollary \ref{P-well-defined}) Moreover, $\n(P_j^{q_j})=\n(P^{L(j)})$, $\l_j\in \k\setminus\{0\}$ and $q_j\b_j<\b_{j+1}<\infty$ in this case.

\end{remark}

Let $T_1=z$, $m_0=1$, $\d_0=1$, $\T_1=\P$, $H_0=\{0\}$ and
$U_0=\{0\}$. For all $i\ge 1$  we set  $\g_i=\n(T_i)$ if $T_i\neq 0$ and $\g_i=0$ if $T_i=0$,  also set  
$H_i=<\g(i)>$, $U_i=[\g(i)]$.

For all $i\ge 1$ such that $T_i\neq 0$, if $\g_i$ is commensurable  with $ G+H_{i-1}$, set
\begin{center}
$
\begin{array}{rl}
s_i =&\min\{s\in\N\,\mid s\g_i\in (G+H_{i-1})\}\\
m_i =&\max(m_{i-1}, \min\{j\in \N \mid s_i\g_i\in(G_j+H_{i-1})\})\\
\D_i =&\{(A,C)\in\N_0^{m_i}\times\N_0^i \mid  (A,C)
\text{ is irreducible with respect to }  \T_i\\
 & \quad\quad\quad \text{ and }AC \text{ is reduced with
respect to } (\b_1,\dots,\b_{m_i},\g_1,\dots,\g_i)\}\\

\T_{i+1}=&\T_i\cup\D_i\\ 
\end{array}
$
\end{center}
Set $\d_i=\#\D_i$. For $(A,C)\in\D_i$ define
$|(A,C)|=A\cdot\b(m_i)+C\cdot\g(i)$. We
assume that $\D_i=\{(A,C)_1,\dots,(A,C)_{\d_i}\}$ is an ordered set
with $|(A,C)_1|\le\dots\le|(A,C)_{\d_i}|$.
If $(A,C)=(A,C)_k$ for some $k$,  let
$L(AC)\in\N_0^{m_i}$ and $N(AC)\in\N_0^{i-1}$ be such 
that $(L(AC),N(AC))$ is irreducible with respect to $\T_i$ and
$|(A,C)|=|(L(AC),N(AC))|$, where $|(L(AC),N(AC))|$ is defined to be $L(AC)\cdot\b(m_i)+N(AC)\cdot\g(i-1)$.
 Denote by $\m_{AC}$ the
residue of $P^AT^C/P^{L(AC)}T^{N(AC)}$
in $V/m_V$ and set
$$
T_{(\sum_{j=0}^{i-1}\d_j)+k}=T_{AC}=P^AT^C-\m_{AC}P^{L(AC)}T^{N(AC)}
$$

For all $i\ge 1$ such that $\g_i$ is not commensurable  with $ G+H_{i-1}$ we set $s_i=\infty$, $m_i=m_{i-1}$, $\D_i=\emptyset$, $\d_i=0$ and $\T_{i+1}=\T_i$. For all $i>1$ such that $T_i=0$  we set $s_i=1$, $m_i=m_{i-1}$, $\D_i=\emptyset$, $\d_i=0$ and $\T_{i+1}=\T_i\cup\{(\bar{0},E_i)\}$. 

Finally set $H=\bigcup_{i\ge 0} H_i$,  $U=\bigcup_{i\ge 0}U_i$
and $\T=\bigcup_{i\ge 0}\T_i$.

\begin{remark}
The sequence $\{T_j\}_{j>0}$ is well defined since for every $j\ge 1$ such that  $\g_j$ is commensurable with $G+H_{j-1}$ we have $\d_j<\infty$, and for every  $(A,C)\in\D_j$ there exist unique $L(AC)\in\N_0^{m_j}$ and $N(AC)\in\N_0^{j-1}$  such that $(L(AC),N(AC))$ is irreducible with respect to $\T_j$ and
$|(A,C)|=|(L(AC),N(AC))|$.  (See discussion preceding Lemma \ref{finiteness},  Lemma \ref{semigroup representation} and Lemma \ref{semigroup uniqueness})
 Moreover,  $\n(P^AT^C)=\n(P^{L(AC)}T^{N(AC)})$, $\m_{AC}\in \k\backslash\{0\}$ and $|(A,C)|=\n(P^AT^C)<\n(T_{AC})$ in this case.
\end{remark}

For $i,j>0$ we say that $T_j$ is an {\it immediate successor} of $T_i$  if there exists $(A,C)\in\D_i$ such that $T_j=T_{AC}$. We say that $T_j$ is a {\it successor} of $T_i$ if there exists  $k\in\N$ and a sequence of integers $l_0,l_1,\dots,l_k$ such that $l_0=i$, $l_k=j$ and  $T_{l_t}$ is an immediate successor of $T_{l_{t-1}}$ for all $1\le t\le k$. 

\section{Properties of jumping polynomials}\label{properties1}

The goal of this section is to justify the  construction of jumping polynomials. All notations are as introduced in section \ref{construction}.

Let $k\in\N$ and $i\in\N_0$ be fixed integers. We consider the set $\Z^k\times\Z^i$, where we adopt the notation $\Z^k\times\Z^0$ for $\Z^k$. We say that $(A,C)\in\Z^k\times\Z^i$ is a permissible coefficient vector if $0\le a_j<q_j$ for all $j$ such that $q_j<\infty$ and $0\le c_t<s_t$ for all $t$ such that $s_t<\infty$. We also use the notation $|(A,C)|=A\cdot\b(k)+C\cdot\g(i)$.

\begin{lemma}\label{group uniqueness}
Suppose that $k\in\N$ and $i\in\N_0$.  If $(A,C)\in\Z^k\times\Z^i$ and $(B,D)\in\Z^k\times\Z^i$ are permissible coefficient vectors such that $|(A,C)|=|(B,D)|$ then $(A,C)=(B,D)$.
\end{lemma}

\begin{proof}

 Assume first that $i>0$ and $C\neq D$. Let $j=\max\{t\in\N\mid c_t\neq d_t\}$. Then, without loss of generality, we have $c_j<d_j$, and therefore, $0<d_j-c_j<s_j$. Since $c_t=d_t$ for all $j<t\le i$ the equality $|(A,C)|=|(B,D)|$ implies 
  $$(d_j-c_j)\g_j=\sum_{t=1}^{k}(a_t-b_t)\b_t+\sum_{t=1}^{j-1}(c_t-d_t)\g_t$$
So, $(d_j-c_j)\g_j\in(G+H_{j-1})$. This contradicts the definition of $s_j$ as the minimal positive integer $s$ such that $s\g_j\in (G+H_{j-1})$. Thus, $C=D$.
 
 Assume now that $i\ge 0$, $C=D$ and $A\neq B$.  Let $j=\max\{t\in\N\mid a_t\neq b_t\}$. Then, without loss of generality, we have $a_j<b_j$,  and therefore,  $0<b_j-a_j<q_j$. Then since $C=D$ and $a_t=b_t$ for all $j<t\le m$ the equality $|(A,B)|=|(C,D)|$ implies 
 $$(b_j-a_j)\b_j=\sum_{t=1}^{j-1}(a_t-b_t)\b_t$$
So, $(b_j-a_j)\b_j\in G_{j-1}$. This contradicts the definition of $q_j$ as the minimal positive integer $q$ such that $q\b_j\in G_{j-1}$. Thus, $A=B$.
\end{proof}

\begin{lemma}\label{group representation}
Suppose that $\a\in(G_k+H_i)$ for some $k\in\N$ and $i\in\N_0$. Then there exists a unique  permissible coefficient vector $(L,N)\in\Z^m\times\Z^i$
such that  $m=max(k,m_i)$ and $\a=|(L,N)|$.
\end{lemma}

\begin{proof} 
We first observe that if $k\le m_i$ then $\a\in(G_k+H_i)$ implies
$\a\in(G_{m_i}+H_i)$ whereas the conclusion of the lemma does not depend on $k$.
Thus, it suffices to show that the statement holds for $k, i\in\N_0$ such that  $k\ge m_i$.

We will use induction on $(i,k)$ to prove existence of a permissible coefficient vector $(L,N)\in\Z^k\times\Z^i$
such that  $\a=|(L,N)|$, while uniqueness will follow at once from  Lemma \ref{group uniqueness}. If $i=0$ and $k=1$ the statement is trivial as $q_1=\infty$. 

Assume that $i$ is fixed and $k>m_i$. Since $\a\in (G_k+H_i)$ there exists a coefficient vector $(B,D)\in\Z^k\times\Z^i$ such that $\a=|(B,D)|=B\cdot\b(k-1)+b_k\b_k+D\cdot\g(i)$. If $q_k=\infty$ set $l_k=b_k$, otherwise set $l_k$ equal to the residue of $b_k$ modulo $q_k$. Then $(b_k-l_k)\b_k\in G_{k-1}$. So, $\a-l_k\b_k=B\cdot\b(k-1)+(b_k-l_k)\b_k+D\cdot\g(i)$ is an element of  $G_{k-1}+H_i$, and there exists a permissible coefficient vector $(L',N)\in\Z^{k-1}\times\Z^i$ such that   $\a-l_k\b_k=|(L',N)|$. We set $L=(l'_1,l'_2,\dots,l'_{k-1},l_k)$ to get a permissible  coefficient vector $(L,N)\in\Z^k\times\Z^i$ such that $\a=l_k\b_k+L'\cdot\b(k-1)+N\cdot\g(i)=|(L,N)|$.

Assume now that $i\ge 1$ and $k=m_i$. Since $\a\in (G_{m_i}+H_i)$ there exists a coefficient vector $(B,D)\in\Z^{m_i}\times\Z^i$ such that $\a=|(B,D)|=B\cdot\b(m_i)+D\cdot\g(i-1)+d_i\g_i$.  If $s_i=\infty$ set $n_i=d_i$, otherwise set $n_i$ equal to the residue of $d_i$ modulo $s_i$.  Then $(d_i-n_i)\g_i\in (G_{m_i}+H_{i-1})$. So, $(\a-n_i\g_i)=B\cdot\b(m_i)+D\cdot\g(i-1)+(d_i-n_i)\g_i$ is an element of $G_{m_i}+H_{i-1}$, and since $m_i\ge m_{i-1}$ there exists a permissible coefficient vector $(L,N')\in\Z^{m_i}\times\Z^{i-1}$ such that   $\a-n_i\g_i=|(L,N')|$. We set $N=(n'_1,n_2,\dots,n'_{i-1},n_i)$ to get a permissible  coefficient vector $(L,N)\in\Z^{m_i}\times\Z^i$ such that $\a=L\cdot\b(m_i)+N'\cdot\g(i-1)+n_i\g_i=|(L,N)|$.

\end{proof}

We will now verify that  the sequence $\{P_j\}_{j> 0}$ is well defined.  Observe first that if $\{P_j\}_{j> 0}$ is infinite then $q_k<\infty$ for all $k>1$, otherwise $q_k<\infty$ for all $1<k<i$ where $P_i$ is the last element of $\{P_j\}_{j> 0}$.  Also, $\b_k>0$ for all $k\ge 1$, and  $\b_k>q_{k-1}\b_{k-1}$ for all $k\ge 3$. Moreover,  if $k\in\N$ is such that $q_k<\infty$ and $A\in\Z^k$ is a permissible coefficient vector then $\sum_{t=2}^k a_t\b_t<q_k\b_k$. Indeed, since $0\le a_2<q_2$  the inequality holds for $k=2$. If $k>2$ then $q_{k-1}<\infty$ and assuming that the inequality holds for $k-1$ we have
$\sum_{t=2}^k a_t\b_t<q_{k-1}\b_{k-1}+a_k\b_k<\b_k+a_k\b_k\le q_k\b_k$.
Finally,  observe that for $A\in\N_0^k$ the following statements are equivalent: 
\begin{enumerate}
\item $A$ is a permissible coefficient vector.
\item $(A,\bar{0})$ is irreducible with respect to $\P_{k+1}$.
\item $(A,\bar{0})$ is irreducible with respect to $\T$.
\end{enumerate}

\begin{corollary}\label{P-well-defined} 
Let $k\in\N$ be such that $q_k<\infty$. If $\a\in G_k$ satisfies the inequality $\a\ge q_k\b_k$ then there exists a unique $L\in\N_0^k$ such that $(L,\bar{0})$ is irreducible with respect to $\T$ and  $\a=L\cdot\b(k)$.

In particular, there exists a unique $L(k)\in\N_0^{k-1}$ such that $(L(k),\bar{0})$ is irreducible with respect to $\P_k$ and $q_k\b_k=L(k)\cdot\b(k-1)$. 

\end{corollary}

\begin{proof} 
We apply Lemma \ref{group representation} to find  the unique permissible vector  $L\in\Z^k$ such that  $\a=L\cdot\b(k)$. To verify the statement of the corollary it suffices to show that $l_1\ge 0$. From the above discussion it follows that 
$ l_1\b_1=\a-\sum_{t=2}^k l_t\b_t>\a-q_k\b_k\ge 0$, and therefore, $l_1>0$.

In particular, if $k>2$ we apply the statement of the corollary to $\a=q_k\b_k$ , an element of $G_{k-1}$. By construction we have $q_{k-1}<\infty$ and $q_k\b_k\ge\b_k>q_{k-1}\b_{k-1}$, so there exists a unique $L(k)\in\N_0^{k-1}$ such that $(L(k),\bar{0})$ is irreducible with respect to $\P_k$ and $q_k\b_k=L(k)\cdot\b(k-1)$. For $k=2$ observe that $q_2\b_2=l\b_1$ for some $l\in\Z$. Since $q_2,\b_2,\b_1>0$ it follows that $l>0$. Thus $L(2)=(l)$ is the required vector.  \end{proof}

For each $j>1$ the jumping polynomial  $P_j$ is a nonzero monic polynomial in $y$ with coefficients in $\k[x]$ and the degree of $P_j$ is $\prod_{t=2}^{j-1} q_t$. Indeed, $P_2=y$ is a monic polynomial of degree 1. If  $j>2$ then 
$P_j=P_{j-1}^{q_{j-1}}-\l_{j-1}P^{L(j-1)}$.  Since $\deg_y P_{j-1}^{q_{j-1}}=q_{j-1}\prod_{t=2}^{j-2} q_t$ and  $\deg_y P^{L(j-1)}=l_2+l_3q_2+l_4q_2q_3+\dots+l_{j-2}q_2\cdots q_{j-3}
<q_2+(q_3-1)q_2+\dots+(q_{j-2}-1)q_2\cdots q_{j-3}= \prod_{t=2}^{j-2} q_t$, the degree and leading term of $P_j$ is determined by $P_{j-1}^{q_{j-1}}$.  Since $P_{j-1}$ is a monic polynomial in $y$ it follows that  $P_j$ is also a monic polynomial in $y$ and the degree of $P_j$ is $\prod_{t=2}^{j-1} q_t$.

Our next step is to verify that if $T_i\neq 0$ and $\g_i$ is commensurable  with $G+H_{i-1}$ the set $\D_i$ is nonempty and has  finitely many elements.

\begin{lemma}\label{semigroup representation}

Suppose that $\a\in(S_k+U_i)$ for some $k\in\N$ and $i\in\N_0$. Then there exists $(L,N)\in\N_0^m\times\N_0^i$
such that  $m=max(k,m_i)$,  $(L,N)$ is irreducible with respect
to $\T$ and $\a=|(L,N)|$.
\end{lemma}

\begin{proof}
We first observe that if $i$ is fixed and $k\le m_i$ then $\a\in(S_k+U_i)$ implies
$\a\in(S_{m_i}+U_i)$ whereas the conclusion of the lemma does not depend on $k$.
Thus, it suffices to show that the statement holds for $k, i\in\N_0$ such that  $k\ge m_i$.

We use induction on $(i,k)$. If $i=0$ and $k=1$ the statement is trivial. 

Suppose that $i$ is fixed and $k>m_i$. Since $\a\in (S_k+U_i)$ there exists $(B,D)\in\N_0^k\times\N_0^i$ such that $\a=|(B,D)|$. If $q_k=\infty$ set $l_k=b_k$, otherwise set $l_k$ equal to the residue of $b_k$ modulo $q_k$. Then $(b_k-l_k)\b_k$ is a nonnegative multiple of $q_k\b_k$,  so $(b_k-l_k)\b_k\in S_{k-1}$ by Corollary \ref{P-well-defined}. Since  $\a-l_k\b_k=B\cdot\b(k-1)+(b_k-l_k)\b_k+D\cdot\g(i)$ it follows that $(\a-l_k\b_k)\in (S_{k-1}+U_i)$, and there exists  $(L',N)\in\N_0^{k-1}\times\N_0^i$ such that  $(L',N)$ is irreducible with respect to $\T$ and  $\a-l_k\b_k=|(L',N)|$. We set $L=(l'_1,l'_2,\dots,l'_{k-1},l_k)$ to get $(L,N)\in\N_0^k\times\N_0^i$ such that $\a=L'\cdot\b(k-1)+l_k\b_k+N\cdot\g(i)=|(L,N)|$. It remains to check that $(L,N)$ is irreducible with respect to $\T$. Let $(A,C)\in\N_0^k\times\N_0^i$ be such that $A\preceq L$ and $C\preceq N$. If $a_k=0$ then $A\preceq L'$, and therefore, $(A,C)\notin \T$ as $(L',N)$ is irreducible with respect to $\T$. Assume now that $a_k>0$. Then,  since $0<a_k\le l_k<q_k$, it follows that $(A,C)\notin \P$. Also, since $k>m_j$ for all $j\le i$  and $a_k>0$, it follows that $(A,C)\notin \D_j$ for all $j\le i$. Finally, since $c_j=0$ for all $j>i$, it follows that $(A,C)\notin \D_j$ for all $j>i$. Thus, $(A,C)\notin\T$ when $a_k>0$. Hence, $(L,N)$ is irreducible with respect to $\T$.

Suppose now that $i>0$ and $k=m_i$. Since $\g_i$ generates $(S_{m_i}+U_i)$ over
$(S_{m_i}+U_{i-1})$  we can write
$$(S_{m_i}+U_i)=\bigcup_{j\ge 0}(j\g_i+S_{m_i}+U_{i-1})$$
 Let $n_i$ be the minimal nonnegative integer $j$ such that
$\a\in(j\g_i+S_{m_i}+U_{i-1})$. Then $\a-n_i\g_i\in(S_{m_i}+U_{i-1})$, and since $m_i\ge m_{i-1}$ there exists  $(L,N')\in\N_0^{m_i}\times\N_0^{i-1}$ such that  $(L,N')$ is irreducible with respect to $\T$ and  $\a-n_i\g_i=|(L,N')|$. We set $N=(n'_1,n'_2,\dots,n'_{i-1},n_i)$ to get $(L,N)\in\N_0^{m_i}\times\N_0^i$ such that $\a=L\cdot\b(m_i)+N'\cdot\g(i)+n_i\g_i=|(L,N)|$. It remains to check that $(L,N)$ is irreducible with respect to $\T$. Assume for contradiction that $(A,C)\in\N_0^{m_i}\times\N_0^i$ is such that $A\preceq L$, $C\preceq N$ and $(A,C)\in\T$. If $c_i=0$ then $C\preceq N'$ which contradicts the assumption that  $(L,N')$ is irreducible with respect to $\T$. Thus, $c_i>0$. Since  $c_i>0$  and $c_j=0$ for all $j>i$, $(A,C)\in\T$ implies that $(A,C)\in\D_i$. Thus, $AC$ pushes $\g_i$ into $(S_{m_i}+U_{i-1})$, that is $|(A,C)|\in (S_{m_i}+U_{i-1})$. We write 
$\a=|(L-A,N-C)|+|(A,C)|=(n_i-c_i)\g_i+(L-A)\cdot \b(m_i)+(N-C)\cdot\g(i-1)+|(A,C)|$ to conclude that $\a\in((n_i-c_i)\g_i+S_{m_i}+U_{i-1})$. Since $0\le n_i-c_i<n_i$, this is a contradiction to the choice of $n_i$. Hence, $(L,N)$ is irreducible with respect to $\T$.
\end{proof}

  Assume that $i\in\N$ is such that  $T_i\neq 0$ and $\g_i$ is commensurable with $G+H_{i-1}$. Then by construction $\g_i$ is commensurable with $G_{m_i}+H_{i-1}$. So, there exists  $(A,C)\in \N_0^{m_i}\times\N_0^i$ such that $AC$ pushes $\g_i$ into $S_{m_i}+U_{i-1}$.  Moreover, since $A\cdot\b(m_i)+C\cdot\g(i-1)\in (S_{m_i}+U_{i-1})$, by Lemma \ref{semigroup representation} there exists $(L,N')\in\N_0^{m_i}\times\N_0^{i-1}$ such that $|(L,N')|=A\cdot\b(m_i)+C\cdot\g(i-1)$ and $(L,N')$ is irreducible with respect to $\T$. Setting $N=(n'_1,n'_2,\dots,n_{i-1},c_i)$ we obtain a coefficient vector $(L,N)\in\N_0^{m_i}\times\N_0^i$ which is irreducible with respect to $\T_i$ with $|(L,N)|=|(A,C)|$.
Thus, there exists $(L,N)\in\N_0^{m_i}\times\N_0^i$ such that $(L,N)$ is irreducible with respect to $\T_i$ and  $LN$ pushes $\g_i$ into $S_{m_i}+U_{i-1}$.

Consider the set $\Y_i$ of all  $(L,N)\in\N_0^{m_i}\times\N_0^i$ such that $(L,N)$ is irreducible with respect to $\T_i$ and  $LN$ pushes $\g_i$ into $S_{m_i}+U_{i-1}$. By the above discussion $\Y_i$ is nonempty. Also,  by abuse of notation $\Y_i$ can be viewed as a subset of $\N_0^{m_i+i}$. Then, for $LN\in\Y_i$, $LN$ is reduced with respect to $(\b(m_i),\g_i)$ if and only if  $LN$ is a minimal element of $\Y_i$. Hence, $\D_i$ is the set of minimal elements of a nonempty subset $\Y_i\subseteq \N_0^{m_i+i}$, and therefore, is nonempty. We observe that $\D_i$ cannot contain a strictly increasing sequence of elements as every element but the first one of such a sequence will not be minimal. It will follow that $\D_i$ is  finite from the lemma below.

\begin{lemma}\label {finiteness} Let $n$ be a positive integer. If $\X\subseteq \N_0^n$ is an infinite set then there exists a strictly increasing sequence $\S=\{A_k\}_{k=1}^{\infty}$ in $\X$.
\end{lemma}
\begin{proof}
We will argue by induction on $n$. For $n=1$, since $\X$ is an infinite subset of $\N_0$, setting $\S$ equal to the ordered set of all elements of $\X$ proves the statement.  

Let $n>1$. If $\pi_{\hat{1}}(\X)$  and $\pi_{\hat{2}}(\X)$ were both finite, then $\pi_1(\X)=\pi_1(\pi_{\hat {2}}(\X))$ would be finite and, therefore,  $\X$ would be finite as a subset  of $\pi_1(\X)\times\pi_{\hat{1}}(\X)$. Thus, there exists $i\in\{1, 2\}$ such that $\pi_{\hat{i}}(\X)$ is an infinite subset of $\N_0^{n-1}$. Then by the inductive hypothesis we can find a strictly increasing sequence $\cR'=\{A'_k\}_{k=1}^{\infty}$ in $\pi_{\hat{i}}(\X)$. Let $\cR=\{A_k\}_{k=1}^{\infty}$ be a sequence in $\X$ such that $\pi_{\hat{i}}(A_k)=A'_k$ for every $k\in\N$.
Then $\pi_i(\cR)$ is a subset of $\N_0$ that may be finite or infinite.  We will consider these two cases separately.

 Assume first that $\pi_i(\cR)$ is finite. Then since $\cR$ is infinite, there exists an element $a\in \pi_i(\cR)\subseteq\N_0$ such that infinitely many elements of $\cR$ are mapped into $a$. Denote by $\S=\{A_{k_j}\}_{j=1}^{\infty}$ the subsequence $\cR\cap \pi_i^{-1}(\{a\})$. Then, for $1\le j<l$, we have 
 $\pi_i(A_{k_j})=\pi_i(A_{k_l})=a$,  $\pi_m(A_{k_j})\le \pi_m(A_{k_l})$ for all  $m\in\{1,\dots,n\}\setminus\{i\}$  since $\pi_{\hat{i}}(A_{k_j})\prec \pi_{\hat{i}}(A_{k_l})$, and $A_{k_j}\neq A_{k_l}$ since  $\pi_{\hat{i}}(A_{k_j})\neq \pi_{\hat{i}}(A_{k_l})$. Thus $\S$ is a  strictly increasing sequence in $\X$.
 
Assume now that $\pi_i(\cR)$ is infinite. Then the sequence $\{\pi_i(A_k)\}_{k=1}^{\infty}$ is unbounded, and therefore, there exists a strictly increasing subsequence $\S'=\{\pi_i(A_{k_j})\}_{j=1}^{\infty}$. Set $\S=\{A_{k_j}\}_{j=1}^{\infty}$. Then, for $1\le j<l$, we have $\pi_{i}(A_{k_j})<\pi_{i}(A_{k_l})$, and $\pi_m(A_{k_j})\le \pi_m(A_{k_l})$ for all  $m\in\{1,\dots,n\}\setminus\{i\}$ since $\pi_{\hat{i}}(A_{k_j})\prec \pi_{\hat{i}}(A_{k_l})$. Thus $\S$ is a  strictly increasing sequence in $\X$.

\end{proof}

The next two statements will allow us to deduce that for every $i\in\N$ such that  $T_i\neq 0$ and $\g_i$ is commensurable with $G+H_{i-1}$ and every $(A,C)\in\D_i$ there exists a unique $(L(AC),N(AC))\in\N_0^{m_i}\times\N_0^{i-1}$ that  satisfies the equality $|(A,C)|=|(L(AC),N(AC))|$ and is irreducible with respect to $\T_i$. Since $|(A,C)|\in (S_{m_i}+U_{i-1})$ existence of the required coefficient vector $(L(AC),N(AC))$ follows from Lemma \ref{semigroup representation}. We will now verify uniqueness. 

\begin{lemma}\label{auxiliary}
Suppose that $k\in\N$ and $i\in\N_0$ are such that $k>m_i$. If  $\a\in G_k$ is such that $\a\in (G_{k-1}+H_i)$ then $\a\in G_{k-1}$.
\end{lemma}

\begin{proof} We will argue by induction on $i$. If $i=0$ the statement is trivial. 

Assume that $i\ge 1$. Since $\a\in (G_{k-1}+H_i)$ and $\g_i$ generates $G_{k-1}+H_i$ over $G_{k-1}+H_{i-1}$ we can write $\a=\b+t\g_i$, where $t\in\Z$ and $\b\in (G_{k-1}+ H_{i-1})$. Then $t\g_i=\a-\b$  and since $\a\in G_k\subseteq G$ it follows that $t\g_i\in( G+H_{i-1})$. Thus, $t$ is a multiple of $s_i$, and  therefore,  $t\g_i\in  (G_{m_i}+ H_{i-1})$. Since $ m_i\le k-1$ it follows that $t\g_i\in (G_{k-1}+ H_{i-1})$, and therefore, $\a\in (G_{k-1}+ H_{i-1})$. Then  the inductive hypothesis implies that $\a\in G_{k-1}$.
\end{proof}

\begin{lemma}\label{semigroup uniqueness}
Suppose that $k\in\N$ and $i\in\N_0$.  If $(A,C)\in\N_0^k\times\N_0^i$ and $(B,D)\in\N_0^k\times\N_0^i$ are irreducible with respect to $\T_{i+1}$ and $|(A,C)|=|(B,D)|$ then $(A,C)=(B,D)$.
\end{lemma}

\begin{proof} Assume for contradiction that $(A,C)\neq (B,D)$. After possibly changing notation we may assume that $a_k\neq b_k$ and, provided $i>0$,  $c_i<d_i$.

If $k>m_i$ or $i=0$ the equality $|(A,C)|=|(B,D)|$ implies that 
$$(a_k-b_k)\b_k=(B-A)\cdot\b(k-1)+(D-C)\cdot\g(i)$$
where the left hand side is an element of $G_k$ and the right hand side is an element of $G_{k-1}+H_i$. Then it follows from Lemma \ref{auxiliary} that $(a_k-b_k)\b_k\in G_{k-1}$, and therefore, $a_k-b_k$ is a multiple of $q_k$. However, since $(A,C)$ and $(B,D)$ are irreducible with respect to $\P$ we have $0\le a_k,b_k<q_k$, and therefore, $0<|a_k-b_k|<q_k$, as $a_k\neq b_k$. This is a contradiction.

If  $0<i$ and $k\le m_i$ set $X\in \N_0^{m_i}$ to be the vector such that $x_t=\max(b_t-a_t,0)$ if $1\le t\le k$ and $x_t=0$ if $k<t\le m_i$,  and set $Y\in \N_0^i$ to be the vector such that $y_t=\max(d_t-c_t,0)$ for all $1\le t\le i$. Also set  $Z\in \N_0^{m_i}$ to be the vector such that $z_t=\max(a_t-b_t,0)$ if $1\le t\le k$ and $z_t=0$ if $k<t\le m_i$,  and $W\in \N_0^{i-1}$ to be the vector such that $w_t=\max(c_t-d_t,0)$ for all $1\le t\le i-1$.  Then the equality $|(A,C)|=|(B,D)|$ implies that 
$$X\cdot\b(m_i)+Y\cdot\g(i)=Z\cdot\b(m_i)+W\cdot\g(i-1)$$
 where the right hand side is an element of $S_{m_i}+U_{i-1}$. Since  $y_i>0$,  $XY$ pushes $(\b(m_i),\g(i))$ into  $S_{m_i}+U_{i-1}$. Denote by $(L,N)$ the minimal vector $(X',Y')\in\N_0^{m_i}\times\N_0^i$ such that $X'Y'$ pushes $(\b(m_i),\g(i))$ into  $S_{m_i}+U_{i-1}$ and $X'Y'\preceq XY$. Then $LN$ is reduced with respect to $(\b(m_i),\g(i))$. Also, since $L\preceq X \preceq B$ and $N\preceq Y \preceq D$ and $(B,D)$ is irreducible with respect to $\T_i$ it follows that $(L,N)$ is irreducible with respect to $\T_i$. Thus, $(L,N)\in\D_i\subseteq \T_{i+1}$.  This contradicts the assumption that $(B,D)$ is irreducible with respect to $\T_{i+1}$. 
\end{proof}

\section{Generating sequences}\label{generating seq}

In this section we will show that the set of jumping polynomials $\{P_j\}_{j>0}\cup\{T_j\}_{j>0}$ forms a generating sequence  of $\n$. Recall that $S_R = \nu(R \backslash \{0\})$ denotes the value semigroup of $R$, and for  $\s \in S_R$, $I_{\s}=\{f\in R \mid \ \n(f)\ge \s\}$  denotes the corresponding valuation ideal of $R$. 
A collection $\{ Q_i \}_{i\in I}$ of elements of $R$ is a {\it generating sequence} of $\n$  if for every $\s\in S_R$ the ideal $I_\s$ is generated by the set $\{Q^B \mid B\in \NN,\ \n(Q^B)\ge \s\}$. 

We observe that if $(A,C)\in\NN\times\NN$ then $\n(P^AT^C)=|(A,C)|$. So, for every $\s\in S_R$  we denote by 
$\A_\s$ the ideal of $R$ generated by  
$\{P^AT^C \mid (A,C)\in \NN\times\NN,\ |{(A,C)}|\ge \s\}$,
 and we denote   by 
$\A^+_\s$ the ideal  generated by $\{P^AT^C \mid (A,C)\in \NN\times\NN,\ |{(A,C)}|> \s\}$. Observe that $\A_\s\subseteq I_\s$. Since $R$ is Noetherian $\A^+_\s$ is finitely generated, and therefore, there exists some $\bar\s\in S_R$ such that $\A^+_\s=A_{\bar\s}$. Observe also that $R=\A_0$ and $m_R=\A_{\b_1}=\A_0^+$. 

To verify that $\{P_j\}_{j> 0}\cup\{T_j\}_{j>0}$ is a generating sequence of $\n$ we will show that $I_\s=\A_\s$ for all $\s\in S_R$.

\begin{lemma}\label{induction-base} Suppose that $k\in\N$, $A\in\N_0^k$ and $\s=A\cdot \b(k)$. Then there exist $B\in\N_0^k$, $\th\in \k$ and $f\in\A^+_\s$ such that  $B\cdot \b(k)=\s$, $(B,\bar{0})$ is irreducible with respect to $\T$, and $P^A=\th P^B+f$. 
\end{lemma}

\begin{proof} We argue by induction on $(k,a_k)$. If $k=1$ the statement holds with $B=A$, $\th=1$ and $f=0$. Assume that $k>1$. We will consider two cases separately:  $a_k<q_k$ and $a_k\ge q_k$.

If $a_k<q_k$ set $A'=\pi_{\hat{k}}(A)$ and  apply the inductive  hypothesis  to $A'$ and $\s'=A'\cdot \b(k-1)$  to obtain $B'\in\N_0^{k-1}$, $\th'\in \k$ and $f'\in\A^+_{\s'}$ such that  $B'\cdot \b(k-1)=\s'$, $(B',\bar{0})$ is irreducible with respect to $\T$, and $P^{A'}=\th'P^{B'}+f'$. Set $B=B'+a_kE_k$, and observe that $(B,\bar{0})$ is irreducible with respect to $\T$ and 
$$B\cdot\b(k)=B'\cdot\b(k-1)+a_k\b_k=\s'+a_k\b_k=A'\cdot\b(k-1)+a_k\b_k=A\cdot\b(k)=\s$$
Also, observe that  $\A^+_{\s'}\A_{a_k\b_k}\subseteq \A^+_{\s'+a_k\b_k}=\A^+_{\s}$. So, we set $\th=\th'$ and $f=f'P_k^{a_k}$ to obtain  $\th\in \k$ and $f\in \A^+_{\s}$ such that $P^A=P^{A'}P_k^{a_k}=(\th'P^{B'}+f')P_k^{a_k}=\th P^B+f$.

If  $a_k\ge q_k$ set $A'=A-q_kE_k+L(k)$ and $A''=A-q_kE_k+E_{k+1}$.   We observe that  $A'\cdot\b(k)=A\cdot\b(k)-q_k\b_k+L(k)\cdot\b(k-1)=\s$ since  $L(k)\cdot\b(k-1)=q_k\b_k$, and    $A''\cdot\b(k+1)=A\cdot\b(k)-q_k\b_k+\b_{k+1}>\s$ since $\b_{k+1}>q_k\b_k$. Moreover,  
$$P^A=P^{A-q_kE_k}P_k^{q_k}=P^{A-q_kE_k}(\l_kP^{L(k)}+P_{k+1})=\l_k P^{A'}+P^{A''},$$
where $\l_k\in \k$, $P^{A''}\in\A^+_{\s}$ and $A'\cdot\b(k)=\s$. Since $A'\in\N_0^k$ and $a'_k=a_k-q_k<a_k$ applying the inductive hypothesis to $A'$ we find $B\in\N_0^{k}$, $\th'\in \k$ and $f'\in\A^+_{\s}$ such that  $B\cdot \b(k)=\s$, $(B,\bar{0})$ is irreducible with respect to $\T$, and $P^{A'}=\th'P^B+f'$.  Set $\th=\th'\l_k$ and $f=\l_kf'+P^{A''}$ to get 
$$P^A=\l_k P^{A'}+P^{A''}=\l_k(\th'P^B+f')+P^{A''}=\th P^B+f,$$
where $\th\in \k$, $f\in \A^+_{\s}$ and $B\in\N_0^{k}$ are as required.
\end{proof}

\begin{lemma}\label{reduced}
Suppose that $k\in\N$, $i\in\N_0$ and $(A,C)\in\N_0^k\times\N_0^i$. Let $\s=|(A,C)|$ and $m=\max(k,m_i)$. Then there exist $(B,D)\in\N_0^m\times\N_0^i$, $\th\in\k$ and $f\in\A^+_\s$ such that  $|(B,D)|=\s$, $(B,D)$ is irreducible with respect to $\T$, and $P^AT^C=\th P^BT^D+f$. 
\end{lemma}

\begin{proof}

We first observe that if $i$ is fixed and $k\le m_i$ then $(A,C)\in\N_0^k\times\N_0^i$ implies
 $(A,C)\in\N_0^{m_i}\times\N_0^i$ whereas the conclusion of the lemma does not depend on $k$.
Thus, it suffices to show that the result holds for $k, i\in\N_0$ such that  $k\ge m_i$. Assume that $k\ge m_i$ and apply induction on $(i,c_i)$. If $i=0$ Lemma \ref{induction-base} verifies the statement.

Assume that $i>0$. Applying the inductive hypothesis to $(A,\pi_{\hat{i}}(C))\in\N_0^k\times\N_0^{i-1}$ and $\s'=|(A,\pi_{\hat{i}}(C))|$ we obtain $(B',D')\in\N_0^k\times\N_0^{i-1}$, $\th'\in \k$ and $f'\in\A^+_{\s'}$ such that  $|(B',D')|=\s'$, $(B',D')$ is irreducible with respect to $\T$, and $P^AT^{\pi_{\hat{i}}(C)}=\th' P^{B'}T^{D'}+f'$. Observe that $\s'=|(A,C)|-c_i\g_i=\s-c_i\g_i$ and $\A^+_{\s'}\A_{c_i\g_i}\subseteq\A^+_{\s'+c_i\g_i}=\A^+_{\s}$. We set $D''=D'+c_iE_i$ and $f''=f'T_i^{c_i}$ so that 
$$P^AT^C=P^AT^{C_{\hat{i}}}T_i^{c_i}=(\th' P^{B'}T^{D'}+f')T_i^{c_i}=\th'P^{B'}T^{D''}+f'',$$
where $\th'\in \k$, $|(B',D'')|=|(B',D')|+c_i\g_i=\s'+c_i\g_i=\s$ and $f''\in\A^+_{\s}$. Thus, if $(B',D'')$ is irreducible with respect to $\T$, the statement holds with $B=B'$, $D=D''$, $\th=\th'$ and $f=f''$. 

Assume that $(B',D'')$ is not irreducible with respect to $\T$. Then, since $(B',\pi_{\hat{i}}(D''))=(B',D')$ is irreducible with respect to $\T$, there exists $(X,Y)\in\T$ such that  $y_i>0$, $X\preceq B'$ and $Y\preceq D''$. Let $j$ be the index such that $T_j=T_{XY}$. Set $(A',C')=(B',D'')-(X,Y)+(L(XY),N(XY))$ and  $(A'',C'')=(B',D'')-(X,Y)+(\bar{0},E_j)$. We observe that  $|(A',C')|=|(B',D'')|-|(X,Y)|+|(L(XY),N(XY))|=\s$ since  $|(L(XY),N(XY))|=|(X,Y)|$, and    $|(A'',C'')|=|(B',D'')|-|(X,Y)|+\g_j>\s$ since $\g_j>|(X,Y)|$. Moreover,
\begin{align*}
P^AT^C  =\th'P^{B'}T^{D''}+f'' &=\th'P^{B'-X}T^{D''-Y}(\m_{XY}P^{L(XY)}T^{N(XY)}+T_j)+f''=\\
& =\th'\m_{XY}P^{A'}T^{C'}+\th'P^{A''}T^{C''}+f'',
\end{align*} where $\th'\m_{XY}\in \k$, $|(A',C')|=\s$ and $\th'P^{A''}T^{C''}+f''\in\A^+_{\s}$. Since $(A',C')\in\N_0^k\times\N_0^i$ and $c'_i=c_i-y_i<c_i$, applying the inductive hypothesis to $(A',C')$ we find $(B,D)\in\N_0^k\times\N_0^i$, $\th''\in \k$ and $f'''\in\A^+_\s$ such that  $|(B,D)|=\s$, $(B,D)$ is irreducible with respect to $\T$, and $P^{A'}T^{C'}=\th'' P^BT^D+f'''$. Set $\th=\th''\th'\m_{XY}$ and $f=\th'\m_{XY}f'''+\th'P^{A''}T^{C''}+f''$ to get
\begin{align*}
P^AT^C  & =\th'\m_{XY}P^{A'}T^{C'}+\th'P^{A''}T^{C''}+f''=\\
 & =\th'\m_{XY}(\th'' P^BT^D+f''')+\th'P^{A''}T^{C''}+f''=\th P^BT^D+f,
\end{align*}
where $\th\in \k$, $f\in \A^+_{\s}$ and $(B,D) \in\N_0^{k}\times\N_0^i$ are as required.
\end{proof}

\begin{theorem}\label{i=a}
If $\s\in S_R$, then $I_\s=\A_\s$. 

Thus, $\{P_i\}_{i>0}\cup\{T_i\}_{i>0}$ is a generating sequence of $\n$.
\end{theorem}

\begin{proof} It suffices to check that if $f\in I_\s$ for some $\s\in S_R$ then $f\in\A_\s$.

Let $f\in I_\s$ and $\W=\{\a\in S_R\mid f\in\A_\a\}$. Observe that $0\in\W$, and  for every $\a\in\W$ the inequality  $\a\le\n(f)\le\s$ holds. Hence, $\W$ is a nonempty subset of $S_R\cap[0,\n(f)]$. Since $S_R\cap[0,\n(f)]$ is a finite set (see Lemma 3, App. 3 of \cite{ZS}), $\W$ has a maximal element. Let $\t=\max\W$. Since $f\in\A_\t$  there exists
a presentation
$$
f=\sum_{e=1}^N g_eP^{A_e}T^{C_e}+f',
$$
where $N\in\N$, $f'\in \A^+_\t$, and $g_e\in R$,  $(A_e, C_e)\in\NN\times\NN$ with $|(A_e,C_e)|=\t$ for each $e$.

For the $N$ elements $(A_e, C_e)$ of $\NN\times\NN$ let $k,i\in \N$ be such that $(A_e, C_e)\in\N_0^k\times\N_0^i$ for every $e$, and assume that $k\ge m_i$. Since $|(A_e,C_e)|=\t$ for each $e$, by  Lemmas \ref{semigroup representation} and \ref{semigroup uniqueness}, there exists a unique $(B,D)\in\N_0^k\times\N_0^i$ such that $(B,D)$ is irreducible with respect to $\T$ and $|(B,D)|=|(A_e,C_e)|$. Applying Lemma \ref{reduced} to every $(A_e,C_e)$  we get $P^{A_e}T^{C_e}=\th_e P^BT^D+h_e$, where $\th_e\in \k$, $h_e\in\A^+_\t$. Thus
$$
f=\left(\sum_{e=1}^N \th_eg_e\right)P^BT^D+\sum_{e=1}^N h_eg_e+f'= gP^BT^D+h,
$$
where $g=(\sum_{e=1}^N\th_e g_e)$ and  $h=\sum_{e=1}^N h_eg_e+f'$. In particular, $h\in\A_{\t}^+$ and $\n(h)>\t$. 

If $g\in m_R$ then
$ gP^BT^D\in\A^+_0\A_\t\subseteq\A^+_\t$, 
and therefore, $f\in\A^+_\t$.
Let $\bar{\t}\in S_R$ be such that $\A^+_\t=\A_{\bar\t}$. Then $\bar{\t}>\t$ and $f\in\A_{\bar\t}$. This contradicts the choice of $\t$ as a maximal element of $\W$. Thus, $g$ must be a unit in $R$. So, $\n(g)=0$, $\t=\n(gP^BT^D)<\n(h)$ and $\n(f)=\n(gP^BT^D)=\t$. Since $f\in I_{\s}$ it follows that  $\t\ge\s$, and therefore, $\A_\t\subseteq\A_\s$. Hence, $f\in\A_\s$.
\end{proof}

The next lemma will identify some reductions in the set $\{P_i\}_{i\ge 0}\cup\{T_i\}_{i>0}$ that preserve the claim of Theorem \ref{i=a}.

\begin{lemma}\label{redundant} Suppose that $J,J'\subseteq \N$. 

Suppose that for every $i\in J$, 
$\b_i\in S_{i-1}$ and $P_i$ can be represented as a finite linear combination of the terms in the form $P^XT^Y$ with $(X,Y)$ irreducible with respect to $\T$. That is, suppose that for every $i\in J$, $\b_i\in S_{i-1}$ and there exist $M_i\in\N$, $\chi_1(i),\dots,\chi_{M_i}(i)\in\k\setminus \{0\}$ and distinct $(X_1(i),Y_1(i)),\dots,(X_{M_i}(i),Y_{M_i}(i))\in \NN\times \NN$ such that
  $P_i=\sum _{e=1} ^{M_i}\chi_e(i) P^{X_e(i)}T^{Y_e(i)}$ and
$(X_e(i),Y_e(i))$ is irreducible with respect to $\T$ for each $e$.

Suppose also that for every $j\in J'$, $\g_j\in S_{k_j}+U_{j-1}$ and $T_j$ can be represented as a finite (possibly trivial) linear combination of the terms in the form $P^ZT^W$ with $(Z,W)$ irreducible with respect to $\T$. That is, suppose that for every $j\in J'$, $\g_j\in S_{m_j}+U_{j-1}$ and there exist $N_j\in\N_0$, $\psi_1(j),\dots,\psi_{N_j}(j)\in\k\setminus\{0\}$ and distinct $(Z_1(j),W_1(j)),\dots,(Z_{N_j}(j),W_{N_j}(j))\in \NN\times \NN$ such that
  $T_j=\sum _{e=1} ^{N_j}\psi_e(j) P^{Z_e(j)}T^{W_e(j)}$ and
$(Z_e(j),W_e(j))$ is irreducible with respect to $\T$ for each $e$.

 Then $\{P_i\}_{i\in\N\setminus J}\cup\{T_j\}_{j\in\N\setminus J'}$ is a generating sequence of $\n$.
\end{lemma}

\begin{proof} For $\s\in S_R$ denote by $A'_{\s}$  the ideal of $R$ generated by 
$$\{P^AT^C\mid (A,C)\in\NN\times\NN, |(A,C)|\ge \s, \forall i\in J, \forall j\in J', a_i=0, c_j=0\}$$
Since $I_{\s}=\A_{\s}$ and $\A'_{\s}\subseteq \A_{\s}$, to prove the statement of the lemma it suffices to verify that every generator of $\A_{\s}$ belongs to $\A'_{\s}$.

Observe first that for $i\in J$, since $\b_i\in S_{i-1}$, $q_i=1$, and therefore, $(E_i,\bar{0})\in\T$. Hence, if $(A,C)\in\NN\times\NN$ is irreducible with respect to $\T$, then $a_i=0$ for all $i\in J$. Similarly, for $j\in J'$, since  $\g_j\in S_{m_j}+U_{j-1}$, $\bar {0}E_j$ is reduced with respect to $(\b_1,\dots,\b_{m_j},\g_1,\dots,\g_j)$, and therefore $(\bar{0},E_j)\in\T$. Hence, if $(A,C)\in\NN\times\NN$ is irreducible with respect to $\T$, then $c_j=0$ for all $j\in J'$.

Moreover, for $i\in J$, since  $(X_1(i),Y_1(i)),\dots,(X_{M_i}(i),Y_{M_i}(i))$ are distinct and irreducible with respect to $\T$, if $e_1\neq e_2$ then $|(X_{e_1}(i),Y_{e_1}(i))|\neq |(X_{e_2}(i),Y_{e_2}(i))|$ as implied by Lemma \ref{semigroup uniqueness}. Since $\n(\chi_{e}(i)P^{X_{e}(i)}T^{Y_{e}(i)})=|(X_{e}(i),Y_{e})|$, it follows that $\n(\chi_{e_1}(i)P^{X_{e_1}(i)}T^{Y_{e_1}(i)})\neq \n(\chi_{e_2}(i)P^{X_{e_2}(i)}T^{Y_{e_2}(i)})$  for all $e_1\neq e_2$.  Hence, in the triangle inequality  $\n(P_i)\ge\min_{1\le e\le M_i} (|(X_{e}(i),Y_{e}(i))|)$ equality holds. In particular,  $|(X_{e}(i),Y_{e}(i))|\ge \b_i$ for each $e$. Similarly, for every $j\in J'$, since  $(Z_1(j),W_1(j)),\dots$, $(Z_{N_j}(j),W_{N_j}(j))$ are distinct and irreducible with respect to $\T$, Lemma \ref{semigroup uniqueness} implies that  $|(Z_{e_1}(j),W_{e_1}(j))|\neq |(Z_{e_2}(j),W_{e_2}(j))|$ for all $e_1\neq e_2$. So, $|(Z_{e}(j),W_{e}(j))|\ge\g_j$ for each $e$.

Combining the two observations above we conclude that if $i\in J$ then $P^{X_{e}(i)}T^{Y_{e}(i)}\in\A'_{\b_i}$ for each $e$, and therefore $P_i\in\A'_{\b_i}$. Similarly, if $j\in J'$ then $T_j\in \A'_{\g_j}$. Also,  if $i\in\N\setminus J$  and $j\in\N\setminus J'$ then $P_i\in\A'_{\b_i}$ and $T_j\in\A'_{\g_j}$ by construction.

Assume that $\s\in S_R$ and $(A,C)\in\NN\times\NN$ is such that $P^AT^C$ is a generator of $\A_{\s}$. Then $|(A,C)|\ge\s$ and there exist $M,N\in\N$ such that $(A,C)\in \N_0^M\times\N_0^N$. Hence, 
$$P^AT^C=\prod _{i=1}^M P_i^{a_i}\prod_{j=1}^N T_j^{c_j}\in\prod_{i=1}^M (\A'_{\b_i})^{a_i}\prod_{j=1}^N(\A'_{\g_j})^{c_j}\subseteq\A'_{\sum_{i=1}^Ma_i\b_i+\sum_{j=1}^N c_j\g_j}= \A'_{|(A,C)|}\subseteq\A'_{\s}$$

\end{proof}

Jumping polynomials that satisfy the assumptions of Lemma \ref{redundant} will be called {\it {redundant}}. 


\section{Example} \label{example}
Suppose that $\k$  is a field and let $\n$ be a $\k$-valuation on  $\k(x,y,z')$ defined by the following values of variables:  $\n(x)=1$, $\n(y)=\sqrt 2$ and $\n(z')=\sqrt{51}-5$. Then $\n$ is a rank 1 rational rank 3 monomial valuation on $\k(x,y,z')$. By abuse of notation we will denote by $\n$ the restriction of $\n$ to any subfield $K$ of $\k(x,y,z')$.

Let $z=\frac{y^2}{x}+\frac{y^5}{x^5}+z'$, then $\n(z)=2\sqrt{2}-1$ and $\k[x,y,z]_{(x,z,y)}$ is a local ring dominated by $\n$. We refer the reader to a \cite{MacL},\cite{V},\cite{Sp+},\cite{Sp+2} and \cite{NS} for the definition and thorough discussion of the properties of key polynomials. For our example the situation is as simple as possible, we have
\begin{itemize}
\item $\{x\}$ is a complete set of key polynomials for $\n$ corresponding to the field extension $\k\hookrightarrow \k(x)$ \item $\{y\}$ is a complete set of key polynomials for $\n$ corresponding to the field extension $\k(x)\hookrightarrow \k(x,y)$ 
\item $\{z-\frac{y^2}{x}-\frac{y^5}{x^5}\}$ is a complete set of key polynomials for $\n$ corresponding to the field extension $\k(x,y)\hookrightarrow \k(x,y,z)$
\end{itemize}
We will  now use the algorithm of section \ref{construction} to produce a generating sequence of $\n$ in $\k[x,y,z]_{(x,y,z)}$.

We have $P_1=x$, $P_2=y$ with $\b_1=1$ and $\b_2=\sqrt 2$. Since $\sqrt 2$ is not commensurable with $\Z$, the sequence of $P_i$'s is finite with $P_2$ being the last element.  Then
$$T_1=z=x^{-1}y^2+x^{-5}y^{-5}+z'=x^{-1}y^2+hvt  \hspace{2cm} \g_1=2\sqrt{2}-1$$

Since $G+H_0=<1,\sqrt{2}>$, we get $s_1=1$, $m_1=2$ and $\D_1=\{(1,0,1)\}$. Observe that since  $m_1=2$  and $G=G_2$, $m_i=2$ for all $i$. Also, the only immediate successor of $T_1$ is 
$$T_2=xz-y^2=x^{-4}y^5+hvt \hspace{5cm} \g_2=5\sqrt{2}-4$$

Since $G+H_1=<1,\sqrt{2}>$, we get $s_2=1$. Since $S+U_1=[1,\sqrt{2},2\sqrt{2}-1]$, we get $\D_2=\{(2,0,0,1), (1,1,0,1), (0,3,0,1)\}$. So, the three immediate successors of $T_2$ are 
$$T_3=x^2T_2-yz^2=x^3z'+hvt \hspace{5cm} \g_3=\sqrt{51}-2$$
$$T_4=xyT_2-z^3=x^2yz'+hvt \hspace{3.9cm} \g_4=\sqrt{51}+\sqrt 2-3$$
$$T_5=y^3T_2-z^4=xy^3z'+hvt\hspace{3.8cm}\g_5=\sqrt{51}+3\sqrt 2-4$$
Observe also that $T_5=zT_4-yT_2^2$.

Since $G+H_2=<1,\sqrt{2}>$, we have $\g_3$ is not commensurable with $G+H_2$, $s_3=\infty$ and $\D_3=\emptyset$. Since $G+H_3=\n(\k(x,y,z))$, $s_i=1$ for all $i\ge 4$. Also, since $S+U_3=[1,\sqrt{2},2\sqrt{2}-1,5\sqrt{2}-4,\sqrt{51}-2]$, we have $\D_4=\{(1,0,0,0,0,1), (0,1,0,0,0,1),\\
 (0,0,2,0,0,1), (0,0,1,0,0,3)\}$. So, the four immediate successors of $T_4$ are 
 $$T_6=xT_4-yT_3=-x^{-6}y^9+hvt\hspace{4.3cm} \g_6=9\sqrt{2}-6$$
$$T_7=yT_4-zT_3=-x^{-7}y^{10}+hvt \hspace{4.0cm} \g_7=10\sqrt{2}-7$$
$$T_8=z^2T_4-xT_2T_3=-x^5(z')^2+hvt\hspace{3.4cm}\g_8=2\sqrt{51}-5$$
$$T_9=zT_4^3-T_2T_3^3=-x^{10}(z')^2+hvt\hspace{3.4cm}\g_9=4\sqrt{51}-10$$
Observe also that $T_6=-z^2T_2$, $T_7=-xT_2^2$ and $T_9=-T_8^2-zT_2^2T_3T_4+y^2z^2T_2^4+z^2T_2^5$.

Since $T_5$, $T_6$ and $T_7$ are explicitly written as linear combinations of the terms $P^AT^C$, where $(A,C)$ is irreducible with respect to $\T$, they are redundant and, moreover,  we can see that their successors are also redundant. We still note that their immediate successors are 
$$T_{10}=T_5-zT_4=-yT_2^2\hspace{4.5cm}\g_{10}=11\sqrt 2-8$$
 $$T_{11}=T_6+z^2T_2=0\hspace{6.7cm} \g_{11}=0$$
 $$T_{12}=T_7+xT_2^2=0\hspace{6.8cm} \g_{12}=0$$

 Since $S+U_7=[1,\sqrt{2},2\sqrt{2}-1,5\sqrt{2}-4,\sqrt{51}-2, \sqrt{51}+\sqrt 2-3]$, we have $\D_8=\{(1,0,\dots,0,1), (0,1,0,\dots,0,1), (0,0,1,0,\dots,0,1), (0,0,0,1,3,0,0,0,0,1)\}$.  So, the four immediate successors of $T_8$ are 
 $$T_{13}=xT_8+T_3^2=-x^{-8}y^{13}+hvt \hspace{3.3cm}  \g_{13}=13\sqrt{2}-8$$
$$T_{14}=yT_8+T_3T_4=-x^{-9}y^{14}+hvt \hspace{3.0cm}  \g_{14}=14\sqrt{2}-9$$
$$T_{15}=zT_8+T_4^2=-x^{-10}y^{15}+hvt \hspace{3.0cm} \g_{15}=15\sqrt{2}-10$$
$$T_{16}=T_2T_3^3T_8+T_4^5=x^{15}(z')^3+hvt\hspace{2.8cm}\g_{16}=6\sqrt{51}-15$$
We observe that $T_{13}=-z^4T_2$, $T_{14}=-y^2zT_2^2-zT_2^3$, $T_{15}=-yz^2T_2^2-T_2^2T_3$ and\\
$T_{16}=T_8^3-yzT_2^3T_3^3-T_2^2T_3T_4^3+z^3T_2^4T_3^2-z^7T_2^5-2y^2zT_2^6T_3+zT_2^5T_4^2+yz^3T_2^7-zT_2^7T_3$.

Thus, we see that $T_9$, $T_{10}$, $T_{11}$, $T_{12}$, $T_{13}$, $T_{14}$, $T_{15}$ and $T_{16}$ as well as all their successors  are redundant jumping polynomials.

  
 Hence, $\{x,y,z,T_2,T_3,T_4,T_8\}$ is a generating sequence of $\n$ in $\k[x,y,z]_{(x,y,z)}$. Since the value semigroup is minimally generated by $\{\n(x),\n(y),\n(z),\n(T_2),\n(T_3),\n(T_4),\n(T_8)\}$, this sequence is a minimal generating sequence. So, 
 $$\{x, y, z, xz-y^2, x^3z-x^2y^2-yz^2, x^2yz-xy^3-z^3, -x^5z^2+2x^4y^2z-x^3y^4+2x^2yz^3-2xy^3z^2-z^5 \}$$
 is a minimal generating sequence of $\n$ in $\k[x,y,z]_{(x,y,z)}$.


\begin{thebibliography}{10}


\bibitem{CP} S. D. Cutkosky, O. Piltant, {\it Ramification of valuations}, Adv. Math. 183 (2004), 1-79.

\bibitem{CMT} S.D. Cutkosky, H. Mourtad, and B. Teissier
{\it On the construction of valuations and generating sequences on hypersurface singularities}, 
  arXiv:1904.10702. 

\bibitem{GHK}  L. Ghezzi, H.T. Ha`, O. Kashcheyeva, {\it Toroidalization of generating sequences in dimension two
function fields}, J. Algebra 301 (2) (2006), 838-866

\bibitem{Sp+} F. G. Herrera Govantes, M. A. Olalla Acosta, and M. Spivakovsky, {\it Valuations in algebraic field extensions}, J. Algebra 312 (2007), 1033-1074.

\bibitem{Sp+2} F. G. Herrera Govantes, W. Mahboub, M. A. Olalla Acosta,  and M. Spivakovsky, {\it  Key polynomials for simple extensions of valued fields }, 	arXiv:1406.0657.

\bibitem{K} O. Kashcheyeva, {\it Constructing examples of semigroups of valuations}, J. Pure Appl. Algebra
220 (2016), 3826-3860.

\bibitem{MacL} S. MacLane, {\it A construction for absolute values
in polynomial rings}, Trans. Amer. Math. Soc. 40 (1936), no. 3,
363-395.

\bibitem{NS} J. Novacoski and M. Spivakovsky, {\it Key polynomials and pseudo-convergent sequences}, J. Algebra 495 (2018), 199 -219.

\bibitem{Sp} M. Spivakovsky, {\it Valuations in Functions Fields of
Surfaces}, Amer. J. Math. 112 (1990), 107-156.

\bibitem{T1} B. Teissier, {\it Valuations, deformations and toric geometry},
Valuation Theory and its Applications II, F.-V. Kuhlmann, S.
Kuhlmann, M. Marshall editors, Fields Inst. Comm. 33, Amer. Math.
Soc., Providence, RI, 2003, 361-459.

\bibitem{T2} B. Teissier, {\it Overweight deformations of affine toric varieties and local uniformization},
Valuation Theory in Interaction, EMS Series of Congress Reports (2014), 474?565.

\bibitem{V} M. Vaqui$\acute{\text {e}}$, {\it Extension $d^{,}$une valuation}, Trans. Amer. Math. Soc. 359 (2007), 3439-3481.

\bibitem{ZS} O. Zariski, P. Samuel, {\it Commutative Algebra}, Volume II, Springer, 1960.
\end{thebibliography}
\end{document}